\documentclass[12pt]{article}
\usepackage{amsthm}
\usepackage{amsmath}
\usepackage{cite}
\usepackage{tikz}
\usetikzlibrary{arrows}
\usepackage{amssymb}
\usepackage{enumerate}
\usepackage[margin=1.2in]{geometry}

\usepackage{float}

\newtheorem{theorem}{Theorem}[section]

\newtheorem{proposition}[theorem]{Proposition}
\newtheorem{corollary}[theorem]{Corollary}
\newtheorem{lemma}[theorem]{Lemma}
\newtheorem{definition}[theorem]{Definition}
\newtheorem{example}[theorem]{Example}
\newtheorem{remark}[theorem]{Remark}

\include{epsf}

\begin{document}

\title{On $3$-folds having a holomorphic torus action with $6$ fixed points}

\author{Nicholas Lindsay}

\maketitle

\begin{abstract}
We study $3$-folds with an action of a algebraic torus $T$ and finite fixed point set. In particular, assuming the torus action has (exactly) $6$ fixed points we show that aside from Mori fibre spaces, the topology of such spaces is strongly restricted. For $T= \mathbb{C}^{*}$ we prove that there are two explicit infinite families plus a finite number of exceptional cases. For $T = \mathbb{C}^{*} \times \mathbb{C}^{*}$ there are $2$ exceptional cases which are described explicitly.
\end{abstract}
\tableofcontents

\section{Introduction}

Throughout, we work over $\mathbb{C}$, all varieties are assumed to be projective. 

We study the topology of smooth $3$-folds with an algebraic torus action using Mori's minimal model program. Since this is a rather large class of $3$-folds, one viable approach to this problem is to impose an upper bound on the number of fixed points. The number of fixed points is known to be even by \cite[Theorem 2]{CS}.

 The smallest possible number of fixed points of an action of $\mathbb{C}^*$ on a $3$-fold is $4$, and such $3$-folds are completely classified  into $4$ topological types \cite{KPS,CPS,To2}. These correspond to exactly the smooth Fano $3$-folds having the same Betti numbers as $\mathbb{CP}^3$. After this, the least possible number of fixed points is $6$, which is the subject of this paper. We will be primarily interested in understanding $3$-folds up to the following topological notion of equivalence: \begin{definition}
Two almost complex $6$-manifolds $M_{1},M_{2}$ are called AC-equivalent if there is a diffeomorphism $\phi : M_{1} \rightarrow M_{2}$ such that $\phi^{*}(c_{1}(TM_{2})) =c_{1}(TM_{1})$.
\end{definition} 

This notion of equivalence prescribes the homotopy class of the almost complex structure \cite[Proposition 8]{OV}.  First, we prove the following:

\begin{theorem} \label{newmain}
Let $X$ be a smooth $3$-fold with a holomorphic $\mathbb{C}^{*}$-action with  $6$ fixed points. Then either: \begin{enumerate} \item $X$ is a Mori fibre space.
 \item $X $ is AC-equivalent to $ Bl_{C_{n}}(\mathbb{CP}^{3}) $ where  $C_{n} \subset \mathbb{CP}^{3}$ is a smooth rational curve of degree $n$. 
\item $X $ is AC-equivalent to  $Bl_{C'_{n}}(Q)$ or where $Q$ is the quadric $3$-fold, and $C'_{n} \subset Q$ is a smooth rational curve of degree $n$. 
\item X is AC-equivalent to one of a finite number of exceptional cases. \end{enumerate}
\end{theorem}

The assumption that the action has $6$-fixed points implies that $b_{2}(X)=2$, see Lemma \ref{betti}. We note that the $3$-folds in cases 2. and 3. have a $\mathbb{C}^{*}$-action with $6$ fixed points, in Example \ref{cptc} and Example \ref{quadc} respectively. Mori fibre spaces satisfying the hypotheses of Theorem \ref{newmain} are very reasonable, in particular the base and general fibres are projective spaces, see Section \ref{mfa}. However, we do not classify the exceptional cases completely.

To illustrate the restrictive nature of Theorem \ref{newmain} we formulate the following corollary. We phrase our result in terms of the $\Delta$-invariant, $\Delta(X) \in \mathbb{Z}$ \footnote{See Definition \ref{delta}. $\Delta$ is referred to as the discriminant in \cite{OV}, however this clashes with the unrelated notion of discriminant curve of a conic bundle, hence we just refer to $\Delta$ as the "$\Delta$-invariant" in this paper.}, which is a fundamental invariant of complex $3$-folds with $b_{2}=2$ \cite[Section 5.2]{OV}. The $\Delta$-invariant was also used to study Chern numbers of smooth $3$-folds in \cite{CT}. We prove the following:

\begin{corollary} \label{dis}
There is a constant $K$ such that any smooth $3$-fold $X$ having a holomorphic $\mathbb{C}^{*}$-action with $6$ fixed points, satisfies one of the following two conditions:
\begin{enumerate}
\item $X $  is a conic bundle over $\mathbb{CP}^{2}$, with discriminant curve having degree at most $3$. 
\item $\Delta(X) <K$. 

\end{enumerate}
\end{corollary}

For a rank $2$ vector bundle $E$ over $\mathbb{CP}^{2}$ then, $\Delta(\mathbb{P}(E)) = c_{1}(E)^{2} - 4c_{2}(E)$ \cite[Proposition 17]{OV}. In particular, Case 1. cannot be removed from Corollary \ref{dis} since the $3$-folds $X_{n }= \mathbb{P}(\mathcal{O} \oplus \mathcal{O}(n))$ have a $\mathbb{C}^{*}$-action with isolated fixed points and $\Delta (X_{n}) = n^{2}$. We also note that degree $3$ discriminant curves do occur, see Example \ref{ann}. We use Corollary \ref{dis} to prove a restriction on the cohomology ring of $X$ in Corollary \ref{distop}.

Besides from this Corollary, we are also able to obtain a more restrictive result in the case when the action extends to a rank $2$ action, and is not toric. It turns out in this case, we get a statement up to isomorphism of varieties. Let $Y_{0}$ be the terminal Fano $3$-fold described in 11. of the table of \cite[Theorem 1.1]{BHHN}, let $Y$ be the blow-up of $Y_{0}$ in its singular locus. Let $Q$ be a smooth quadric $3$-fold.  We prove the following

\begin{theorem} \label{twotorus}

Let $X$ be a smooth $3$-fold with a holomorphic $\mathbb{C}^{*} \times \mathbb{C}^{*}$-action with  $6$ fixed points. Assume $X$ is not toric. Then either $X$ is a Mori fibre space or $X$ is isomorphic to either $Y$ or $Bl_{p}(Q)$.
\end{theorem}

To prove Theorem \ref{twotorus}, we apply \cite[Theorem 1.1]{BHHN} to handle extremal contractions which contract a $E=\mathbb{CP}^{2}$ to a $\frac{1}{2}(1,1,1)$ cyclic quotient singularity, or when we contract aquadric cone to a double (cDV) point. The image of an extremal contraction always inherits a $\mathbb{C}^{*} \times \mathbb{C}^{*}$-action, see Blanchard's theorem below. 

Our motivation to study such $3$-folds comes partially from symplectic geometry. In particular, building on work of Tolman \cite{To1}, there have been several recent results regarding non-K\"{a}hler Hamiltonian torus action on compact symplectic $6$-manifolds with $b_{2}=2$; see \cite{GKZ1,LP2} and more recently with $b_{2}$ arbitrary \cite{GKZ2}.  To the authors knowledge, it remains an open question whether there is a closed symplectic $6$-manifold with a Hamiltonian $S^{1}$-action with isolated fixed points which is not diffeomorphic to a K\"{a}hler manifold.  

By Hodge theory a K\"{a}hler manifold with $b_{2}=2$ is projective \cite[Lemma 4.1]{LP2}, and Hamiltonian $S^{1}$-actions preserving the K\"{a}hler structure extend to algebraic $\mathbb{C}^{*}$-actions  (this follows from the main results of \cite{So}).  We consider Theorem \ref{newmain} as a step towards understanding the geography of K\"{a}hler $3$--folds with a $\mathbb{C}^{*}$-action with isolated fixed points, in sight of future comparisons with the symplectic case.

\textbf{Acknowledgments.} I would like to thank Alexander Kuznetsov, Dmitri Panov, Jaros\l{}aw Wi\'{s}niewski and Ziyu Zhang for helpful comments. I would like thank Freidrich Knop for providing the reference \cite[Prop. IV.13.5]{Bo}. I would like to thank an anonymous referee for comments. The earlier versions of this preprint were completed whilst the author was
 at the Institute of Mathematical Science in Shanghaitech University. 
\def\CP{\mathbb{C}{\rm P}}
\def\tr{{\rm tr}\,}
\def\endproof{{$\Box$}}

\section{Preliminaries} \label{pre}

\subsection{Preliminaries about algebraic group actions}

Firstly, we we will recall the a result of Sommese, which shows that certain holomorphic actions on projective varieties are algebraic. Recall that an action of an algebraic group $G$ on a projective variety $X$ is called algebraic if it is given by a morphism $\phi : G \times X \rightarrow X$. The statement is the following:
\begin{theorem}
Suppose that $\mathbb{C}^*$ acts holomorphically on a smooth projective variety, with at least one fixed point. Then, the action is algebraic.
\end{theorem}
\begin{proof}
This follows from the main results of \cite{So}.
\end{proof}

\subsubsection{Suminhiro's Theorem}

We wlill also need the following general fact, which follows immedietely from Suminhiro's theorem.

\begin{lemma} \label{twopoint}
Suppose that $W$ is a complex projective variety with normal singularities and with a $\mathbb{C}^*$-action. Then, for every non-trivial orbit $U$, if $\bar{U}$ denotes the Zariski closure of $U$, then $\bar{U} \setminus U$ consists of two distinct points.
\end{lemma}
\begin{proof}
By Suminhiro's theorem \cite[Theorem 1]{S}, there is a invariant open neighbourhod of the orbit which is $\mathbb{C}^*$ isomorphic to some orbit of a linear $\mathbb{C}^*$-action in projective space. The result follows.
\end{proof}

\subsubsection{The Byalinicki-Birula Decomposition}

The next preliminary result we will need, shows that the assumptions of Theorem \ref{newmain} determine the Betti numbers and fundamental group of the manifold. This is why we are naturally led to explore the ring structure of the cohomology of the underlying $6$-manifold, the most important invariant of which is the $
\Delta$-invariant, see Definition \ref{delta}.

\begin{lemma} \label{betti}
Suppose that $X$ is a smooth projective $3$-fold having a holomorphic $\mathbb{C}^{*}$-action with $6$ fixed points. Then $X$ is rational and simply connected. Furthermore,  $H^{2}(X,\mathbb{Z}) \cong \mathbb{Z}^{2}$ and $b_{3}(X) = 0$.
\end{lemma}
\begin{proof}
By the Bialynicki-Birula decomposition \cite{BB}, $X$ contains a Zariski open subset isomorphic to $\mathbb{A}^3$, in particular it is rational. Since the topological fundamental group is a birational invariant $X$ is simply connected. \cite[Theorem 2]{CS} implies that $b_{2}(X) = 2$, $b_{3}(X) = 0$. Finally, $H^{2}(X,\mathbb{Z}) \cong \mathbb{Z}^{2}$ follows from $b_{2}(X)=2$, $\pi_{1}(X) = \{1\}$, and the universal coefficient theorem.
\end{proof}

\subsubsection{Blanchard's Theorem}

We will need the following simple consequence of Blanchard's theorem. See \cite{Br} for a modern reference and a much more general result. Let $T = (\mathbb{C}^{*})^{k} $ be an algebraic torus. Throughout the paper we assume that group actions are effective unless stated otherwise.
 \begin{theorem} \label{blan}
Suppose that $X$ is a smooth projective variety with a $T$-action with finitely many fixed points. Suppose that $F: X \rightarrow Y$ is a Mori extremal contraction. Then there is a $T$-action on $Y$ having finitely many fixed points and making $F$ equivarient.

 If  $\dim(X)=\dim(Y)$ i.e. $F$ is a divisorial contraction, then the action on $Y$ is effective.
\end{theorem}
\begin{proof}
The existence of the action on $Y$ is Blanchard's theorem. Note that the preimage of each fixed point contains a fixed point, hence the induced action on the base has finite fixed point set. To show the last statement, let $E$ be the contracted divisor.  Then note that any non-trivial $g \in T$ moves a point in $X \setminus E$, hence the corresponding automorphism of $Y$ is non-trivial.
\end{proof}

In addition we will need the following technichal Lemma.
\begin{lemma} \label{divisor}
Suppose that $\phi: 	X  \rightarrow Y$ is a $\mathbb{C}^{*}$-equivariant Mori extremal contraction, and there is a divisor $D \subset X$ such that $\phi(D)$ is a point. Then, consider a non-trivial orbit $O = \mathbb{C}^*.p$, let $\bar{O}$ be its Zariski closure. Then, either $\bar{O} \subset D$ or $\bar{O}$ intersects $D$ in $0$ or $1$ points.
\begin{proof}
Since the map is equivariant $O$ is either contained in $D$ or it is disjoint from it. The closure of any non-trivial orbit $O$ contains two additional points $p_{1}$ and $p_{2}$ by   Lemma \ref{twopoint},. So we just have to rule out the case when these two additional points are contracted to $p$ by $\phi$. But in such a case, there is an orbit in $Y$ whose closure contains only one additional point, contradicting Lemma \ref{twopoint}.
\end{proof}

\end{lemma}

I would like to thank Friedrich Knop for pointing to me to the reference of the following theorem. We will use this to rule out certain types of extremal contractions for smooth $3$-folds with a $\mathbb{C}^*$-action.

\begin{theorem} \cite[Prop. IV.13.5]{Bo}. \label{numberfp}
Suppose that $X$ is an irreducible, complex projective variety with normal singularities. If $\mathbb{C}^*$ acts on $X$ then there are at least $\dim(X) +1$ fixed points.
\end{theorem}

Finally, we state a standard result about holomorphic $\mathbb{C}^*$-actios, which we will need.

\begin{theorem} \label{topeu}
Suppose that $X$ is a smooth projective variety with a $\mathbb{C}^*$-action with finitely many fixed points. Then the number of fixed points is equal to $\chi_{top}(X)$.
\end{theorem}
\begin{proof}
This follows directly from \cite[Theorem 2]{CS}.
\end{proof}

\subsection{The $\Delta$-invariant}
Now we breifly switch to the context of real manifolds. In this subsection we recall a basic invariant of closed $6$-manifolds, encoded in the cubic intersection form, which we refer to as the $\Delta$-invariant. For a coordinate free definition and a comprehensive treatment of this invariant see \cite{OV}.
\begin{definition}  \label{delta}
Let $X$ be a closed $6$-manifold with $\pi_{1}(X)=1$, $b_{3}(X)=0$ and $b_{2}(X)=2$, let $a,b$ be an integral basis of $H^{2}(X,\mathbb{Z}) $  \footnote{Note that $H^{2}(X,\mathbb{Z}) \cong \mathbb{Z}^{2}$ by the vanishing of $\pi_{1}$ and the universal coefficient theorem.}. Then, if we let $a_{0}=\int_{X} a^{3}$, $a_{1}=\int_{X} a^{2}b$, $a_{2}=\int_{X} ab^{2}$ and $a_{3} = \int_{X} b^{3}$. Then define: $$\Delta(X) = (a_{0}a_{3} - a_{1}a_{2})^{2}-4(a_{0}a_{2}-a_{1}^{2})(a_{1}a_{3}-a_{2}^{2}). $$   
\end{definition}
$\Delta(X)$ depends only on the intersection cubic form of $X$ \cite[Section 3.1]{OV}, hence is a topological invariant. Despite its seemingly convoluted appearance, the $\Delta$-invariant is an informative topological invariant, particularly in the context of complex geometry \cite[Section 5.2]{OV}. 

Next, we prove a Lemma which computes the $\Delta$-invariant of the blow up of a complex $3$-fold $X$ with $b_{2}(X) = 1$ along a rational curve of degree $n$.

\begin{lemma} \label{blowupsequence}
Suppose that $X$ is a closed complex $3$-fold such that $ H^{2}(X, \mathbb{Z}) = \mathbb{Z}$ and  $H^{2}(X, \mathbb{Z})$ is generated by a class $h$ satisfying $d =\int_{X} h^3 > 0$. Suppose that $c_{1}(X) = r h$ for some $r>0$. If $C$ is a smooth rational curve in $X$ with $[C] = nh^2$, then we have that $$\Delta( Bl_{C} (X)) =   d^2 ((2-rnd)^2 - 4 n^3 d^2).$$
\end{lemma}
\begin{proof}
Throughout, we abuse notation slightly by not distinguishing between subvarieties and their Poincar\'{e} dual classes. Let $Y = Bl_{C} (X)$ and $\pi : Bl_{C} (X) \rightarrow X$ denoted the map associated to the curve blow up. Let $e$ be the exceptional divisor of the blow-up. By another slight abuse of notation we write $h = \pi^{*}h$. By direct computation:

$$\int_{C} c_{1}(X) = \int_{X} r n  h^3  = rnd .$$

On the other hand, $TX|_C \cong \mathcal{O}(2) \oplus N_C$ where $N_C$ is the normal bundle of $C \cong \mathbb{CP}^1$. Letting $k = \int_{C} c_{1}(N_C)$, the above gives $2+k = rnd$ .

The computation of the following intersection numbers may be computing using the fact that the normal bundle of the exceptional divisor is the tautological bundle $\mathcal{O}(-1)$ (in the sense of projective bundles) and the adjunction formula.  For a reference see \cite[Proposition 14]{OV}):
$$ a_0 = \int_{Y} h^3 =  d.$$
$$ a_{1} = \int_{Y}h^2 e = 0. $$ 
$$ a_{2} =  \int_{Y}h e^2 = -\int_{C}h = -nd. $$
$$ a_{3} =  \int_{Y} e^3 =  -k = 2 - rnd. $$

On the other hand, $\{e,h\}$ is an integral basis of $H^{2}(Y,\mathbb{Z})$. To prove this, consider a general class $a \in H^{2}(Y,\mathbb{Z})$ the class $\pi_{*}(a) = kh$ for some $k \in \mathbb{Z}$.  Then, the class $\pi_{*}(a- kh)=0$  and so $(a- kh)$ is a integer multiple of $e$ \footnote{to show that $e$ is a primitive class, one may use the form of the normal bundle given above}. 	Finally, computing using Definition \ref{delta} gives: 
$$\Delta(Y)= d^2(2-rnd)^2 - 4nd^2(-nd)^2 = d^2 ((2-rnd)^2 - 4 n^3 d^2) .$$ \end{proof}

\begin{corollary}
With the notation as in Lemma \ref{blowupsequence} suppose that there is a sequence of rational curves $C_m$ for $m \geq 1$ such that $[C_{m}]= mh$. Then $\{ \Delta(BL_{C_m}(X)): \; m \geq 1 \}$ is bounded above.
\end{corollary}
\begin{proof}
By Lemma \ref{blowupsequence} $$\Delta(Bl_{C_m}(X)) =   d^2 ((2-rmd)^2 - 4 m^3 d^2) . $$ The right hand side tends to $-\infty$ for fixed $r>0$, $d>0$ and $m \rightarrow \infty$.
\end{proof}

\section{Extremal contractions and $\mathbb{C}^{*}$-actions}
In this section we begin preparing to prove Theorem \ref{newmain}. The idea of the proof is to consider Mori Contractions of $X$,  in particular the possibilities for the image of the exceptional divisor, which will be invariant by some $\mathbb{C}^{*}$-action due to Theorem \ref{blan}.  Firstly, we recall a theorem of Mori, which describes the possible birational extremal contractions on a smooth $3$-fold. 

\begin{theorem} \cite{M} \label{main}
Let $X$ be a smooth $3$-fold such that $K_{X}$ is not nef. Then, $X$ admits an extremal contraction $F: X \rightarrow Y$, associated to an extremal ray $R \subset N_{1}(X,\mathbb{R}) \subset H_{2}(X,\mathbb{R})$. An irreducible curve $C \subset X$ is contracted to a point by $F$ $\iff$ $[C] \in R$.  When $\dim(X)=\dim(Y)$, i.e. $F$ is birational, then extremal contractions are of the $5$ following types:
\begin{enumerate}
\item Blow down to a smooth point.
\item Blow down to a smooth curve.
\item Blow down to a cyclic quotient singularity of type $\frac{1}{2}(1,1,1)$. $(E,N) = (\mathbb{P}^{2},\mathcal{O}(-2))$.
\item Blow down to a double (cDV) point. $(E,N) = (P,\mathcal{O}(-1)|_{P})$ where $P$ is a quadric cone in $\mathbb{CP}^3$.
\item Blow down to an ordinary double point singularity. $(E,N) = (Q,\mathcal{O}(-2)|_{Q})$.
\end{enumerate}
\end{theorem}
Here $Q$ denotes a smooth quadric in $\mathbb{CP}^3$. If $\dim(X)>\dim(Y)$ then $X$ is a called a Mori fibre space, and in this case $Y$ is a smooth, projective variety \cite{M}. We recall that  a smooth variety with a $\mathbb{C}^{*}$-action with isolated fixed points is rational (see Lemma \ref{betti}) so $K_{X}$ cannot be nef. Hence $X$ either is a Mori fibre space or has an extremal contraction of the type $1-5$. In the following Lemma we rule out  type 5. 

\begin{lemma} \label{conres}
Suppose that $X$ is a smooth $3$-fold with a $\mathbb{C}^{*}$-action whose fixed point set consists of $6$ isolated points. Then $X$ cannot have an extremal contraction of type $5$.
\end{lemma}

\begin{proof}

Suppose for a contradiction that $X$ has such an extremal contraction $F : X \rightarrow Y$. By Theorem \ref{blan}, the quadric $Q = \mathbb{P}^1 \times \mathbb{P}^1$ contracted by $F$ is invariant by $\mathbb{C}^{*}$, hence $F(Q) \in Y$ is a fixed point. Since $\chi_{top}(\mathbb{P}^1 \times \mathbb{P}^1) =4$ the number of fixed points of the restriction of the action to $ \mathbb{P}^1 \times \mathbb{P}^1$ is $4$ by Theorem \ref{topeu}.

 Hence the action on $Y$ has exactly $3$ fixed points. Then since $\dim(Y) =3$ we get a contradiction with Theorem \ref{numberfp}. \end{proof}

\section{Analysis of the remaining extremal contractions}
In this section, we prove Theorem \ref{newmain}, we proceed by analysing the extremal contractions of type 1.-3. of Theorem \ref{main}.
\subsection{Blow-down to a smooth point}
Firstly we assume that are our variety $X $ has  $\mathbb{C}^{*}$-action with $6$ fixed points and is a blow-up of a smooth $3$-fold $Y$ in a point. We note there are $4$ possibilities for $Y$. Let $Q$ be a smooth quadric $3$-fold.  
\begin{lemma} \label{pointb}
Suppose that $X$ is a smooth $3$-fold with a $\mathbb{C}^{*}$-action with $6$ fixed points. Suppose furthermore that $X$ is the blow-up of a smooth $3$-fold $Y$ in a point, then $X$ is AC-equivalent to the blow-up of $\mathbb{CP}^{3}$, $Q$, $V_{5}$ or $V_{22}$ in a point.
\end{lemma}
\begin{proof}
By Theorem \ref{blan} $Y$ has a $\mathbb{C}^{*}$-action with isolated fixed points and also $b_{2}(Y)=1$.  Hence, $Y$ is a smooth Fano $3$-fold. Such Fano $3$-folds fall into $4$ deformation families, (see  \cite[Theorem 1.1]{To2}, \cite[Theorem 1.1.2]{KPS}), namely $\mathbb{CP}^{3}$, $Q$, $V_{5}$ or $V_{22}$, in particular four AC-equivalence classes. The result follows. 
\end{proof}

We recall the definitions of $V_{5}$ and $V_{22}$ for convenience. $V_{5}$ is the intersection of the Grassmanian $Gr(2,5) \subset \mathbb{CP}^{9}$ embedded by the Plucker embedding and a codimension $3$ linear subspace. $V_{22}$ is the Mukai-Umemura Fano $3$-fold.

\subsection{ Blow-Down to a smooth curve}
In this section we collect results needed to prove Theorem \ref{newmain}. We analyze the case when the $3$-fold has a Mori extremal contraction which is a blow-down to a smooth curve, $\phi : X \rightarrow Y$. Similarly to the previous section, since $Y$ is Fano with Picard rank $1$ and has a $\mathbb{C}^{*}$-action  there are $4$ possibilities for $Y$ up to AC-equivalence, namely $\mathbb{CP}^{3},Q,V_{5}$ and $ V_{22}$ which we deal with in turn.

We recall that on a given smooth $3$-fold  $Y$ and $C \subset Y$ a smooth curve, $Bl_{C}(Y)$ is determined up to AC-equivalence by the triple $(g_{C},[C],c_{1}(N))$, where $g_{C}$ is the genus of $C$, $[C] \in H_{2}(Y,\mathbb{Z})$ is the homology class of $C$, and $N$ is the normal bundle of $C$ in $Y$. We start by giving an example:

\begin{example} \label{cptc}
Let $n >0$ be an integer. Consider the $\mathbb{C}^{*}$-action on $\mathbb{CP}^{3}$ given by $z.[x_{0}:x_{1}:x_{2}:x_{3}] = [x_{0}: zx_{1} : z^{n-1}x_{2} : z^{n} x_{3}]$. Then there is a the smooth $\mathbb{C}^{*}$-invariant curve $C_{n}$ given parametrically by $[X^{n}: X^{n-1}Y : X Y^{n-1} : Y^{n}]$ for $[X:Y] \in \mathbb{CP}^{1}$, having degree $n$. The induced action on $Bl_{C_{n}}(\mathbb{CP}^{3})$ has isolated fixed points.

\end{example}
We make the following remark about higher rank torus actions on these $3$-folds.

\begin{remark} \label{twoone}
One may check that $C_{n} \subset \mathbb{CP}^{3}$ is invariant by a $\mathbb{C}^{*} \times \mathbb{C}^{*}$-action $\iff$ $n=1$. Hence by Theorem \ref{blan}, the only $3$-fold appearing in Example \ref{cptc} having a $\mathbb{C}^{*} \times \mathbb{C}^{*}$-action is $Bl_{L}(\mathbb{CP}^{3})$ where $L$ is a line.
\end{remark}

\begin{lemma} \label{projp}
Suppose that $C \subset \mathbb{CP}^{3}$ is a curve invariant by a $\mathbb{C}^{*}$-action such that the induced action on $Bl_{C}(\mathbb{CP}^{3})$ has isolated fixed points. Then $Bl_{C}(\mathbb{CP}^{3})$ is AC-equivalent to $Bl_{C_{n}}(\mathbb{CP}^{3})$ for some $n > 0$, where $C_{n}$ is defined in Example \ref{cptc}.
\end{lemma}
\begin{proof}
Let $\phi : X \rightarrow \mathbb{CP}^{3}$ be the inverse of the curve blow-up. We recall from Theorem \ref{blan} that $\phi$ is $\mathbb{C}^{*}$-equivariant with respect to some $\mathbb{C}^{*}$-action on $\mathbb{CP}^{3}$ which must have isolated fixed points. Hence, the blow-up curve $C$ is $\mathbb{C}^{*}$-invariant and since the action on $\mathbb{CP}^{3}$ has isolated fixed points the induced action on $C$ is non-trivial. The only smooth curve with a nontrivial $\mathbb{C}^{*}$-action is $\mathbb{CP}^{1}$ i.e the genus of the curve $g_{C}$ must be $0$. If  $[C] = k H^2\in H_{2}( \mathbb{CP}^{3},\mathbb{Z}) \cong \mathbb{Z}$, then note that $c_{1}(N_{C}) + 2 =  4k$, where $N_{C}$ is the normal bundle of $C$. So all the possibilities for $(g_{C},[C],c_{1}(N_{C})) \in \mathbb{Z}_{\geq 0} \times \mathbb{Z}^{2}$ are exhausted by $C_{n} \subset \mathbb{CP}^{3}$ defined in Example \ref{cptc}. These invariants determine the curve blow-up, up to AC-equivalence.
\end{proof}

The case where $Y$ is a smooth quadric $3$-fold $Q$ is very similar to the above case $Y = \mathbb{CP}^{3}$. We start by giving a sequence of examples.

\begin{example} \label{quadc}
Let $n>0$ be an integer. Consider the quadric $Q =  \{x_{2}^2 - x_{0}x_{4} - 2x_{1}x_{3} = 0\} \subset \mathbb{CP}^{4}$, which is smooth \footnote{This equation of the quadric, convenient for studying torus actions, was found in the answer to the following MO question https://mathoverflow.net/questions/306234/mathbbc-actions-on-fano-3-folds.}. Consider the $\mathbb{C}^{*}$-action on $\mathbb{CP}^{4}$ preserving $Q$ given by $$z. [x_{0}:x_{1}:x_{2}:x_{3}:x_{4}] = [x_{0}: zx_{1}: z^{n+1}x_{2}: z^{2n+1} x_{3} : z^{2n+2} x_{4}] .$$ 
Then the smooth curve $C'_{n} \subset Q$ given parametrically by $[X^{2n+2}: X^{2n+1}Y: X^{n+1}Y^{n+1} : X Y^{2n+1} : Y^{2n+2}]$ for some integer $n \geq 0$ is invariant by the action. The induced action on $Bl_{C'_{n}}(Q)$ has isolated fixed points.

\end{example}

We also make the following remark about higher rank torus actions on these $3$-folds;
\begin{remark} \label{twotwo}
We note $Bl_{L}(Q)$ has a $\mathbb{C}^{*} \times \mathbb{C}^{*}$-action with isolated fixed points,  where $C'_{1}=L$ is an invariant line. On the other hand for $n \geq 2$, $Bl_{C'_{n}}(Q)$ does not have $\mathbb{C}^{*} \times \mathbb{C}^{*}$-action. This follows from the fact that there are only $6$ $\mathbb{C}^{*} \times \mathbb{C}^{*}$-invariant curves in $Q$, all of them lines; and any action on $Bl_{C'_{n}}(Q)$ is inherited from and action on $Q$ preserving $C'_{n}$ by Theorem \ref{blan}.
\end{remark}

\begin{lemma} \label{quadp}
Let $Q$ be a smooth quadric $3$-fold. Suppose that $C \subset Q$ is a smooth curve invariant by a $\mathbb{C}^{*}$-action on $Q$, such that the induced action on $Bl_{C}(Q)$ has isolated fixed points. Then $Bl_{C}(Q)$ is AC-equivalent to $Bl_{C'_{n}}(Q)$ for some $n > 0$, where $C'_n$ is defined in Example \ref{quadc}.
\end{lemma}
\begin{proof}

Let $\phi : X \rightarrow Q$ be the inverse of the curve blow-up. We recall from Theorem \ref{blan} that $\phi$ is $\mathbb{C}^{*}$-equivariant with respect to some $\mathbb{C}^{*}$-action on $Q$ which must have isolated fixed points. Hence, the blow-up curve $C$ is $\mathbb{C}^{*}$-invariant and since the action on $Q$ has isolated fixed points the induced action on $C$ is non-trivial.  The only smooth curve with a nontrivial $\mathbb{C}^{*}$ is $\mathbb{CP}^{1}$ i.e the genus of the curve $g_{C}$ must be $0$. Let  $[C] = k \alpha \in H_{2}(Q,\mathbb{Z}) \cong \mathbb{Z}$, where $\alpha$ is the positive generator then note that $3k=2 + c_{1}(N_{C})$ where $N_C$ is the normal bundle of $C$ in $Q$.   So all the possibilities for $(g_{C},[C],c_{1}(N_{C})) \in \mathbb{Z}_{\geq 0} \times \mathbb{Z}^{2}$ are exhausted by $C'_{n} \subset Q$, defined in Example \ref{quadc}. These invariant determine the curve blow-up, up to AC-equivalence.
\end{proof}

The remaining two cases may be dealt with by a Theorem of Tolman \cite[Theorem 2]{To2}.

\begin{lemma} \label{ex}
There is finitely many AC-equivalence classes of $3$-folds obtained as blow ups of $3$-folds in the families $V_{5}$ and $V_{22}$ in a smooth $\mathbb{C}^{*}$-invariant curve.
\end{lemma}
\begin{proof}
Let $X$ be a $3$-fold in the family $V_{5} $ or $ V_{22}$ and let $H$ be the corresponding Hamiltonian, where the symplectic form is normalised so that $[\omega] = -K_{X}$. By Tolman's theorem \cite[Theorem 2]{To2}, it follows that  $H(X) = [-6,6]$. It follows that $1 \leq -K_{X}.C \leq 12$ for any invariant curve $C$ in $V_{5}$ or $V_{22}$ by the ABBV localisation formula \cite{BV}. Let $\alpha$ be the generator of $H^{2}(X,\mathbb{Z}) = \mathbb{Z}$ with $\alpha^{3}>0$, then $-K_{X} = k\alpha$, where $k=2$ if $X$ is in the family $V_{5}$ and $k=1$ if $X$ is in the family $V_{22}$. Since $-K_{X}.C \leq 12$ this implies there are finitely many possibilities for the homology class of $C$, and since $-K_{X}.C = 2 + c_{1}(N)$, where $N$ is the normal bundle of $C$ in $X$, $-1 \leq c_{1}(N) \leq 10$, so there are finitely many possibilities for the normal bundle up to topological equivalence. Hence, there are finitely many possibilities for the blow-up, up to AC-equivalence.
\end{proof}

We note that bluntly applying the above proof one can obtain that there are at most $12$ AC-equivalence classes of curve blow-ups of $V_{5}$ or $V_{22}$ in $\mathbb{C}^{*}$-invariant smooth curves. A finer analysis is possible using results of \cite{KPS}. We also note that these blow-ups will often coincide with the $3$-folds from Example \ref{cptc} and Example \ref{quadc}. For example the blow up of $V_{5}$ in a line is AC-equivalent to $Bl_{C'_{3}}(Q)$, and the blow-up of $V_{5}$ in a quadric is AC-equivalent to $Bl_{C_{4}}(\mathbb{CP}^{3})$, etc.

To conclude the section, we put together the previous lemmas to get complete analysis of the case of smooth curve blow downs.

\begin{proposition} \label{curveb}
Suppose that $X$ is a smooth $3$-fold with a $\mathbb{C}^{*}$-action with $6$ fixed points. Suppose that $X$ has a Mori extremal contraction $\phi:X \rightarrow Y$, which is the blow-up of a smooth curve in $Y$. Then either X is AC-equivalent to one of the $3$-folds from Example \ref{cptc} or Example \ref{quadc}, or $X$ is $AC$ equivalent to one of a finite number  of (at most $24$) exceptional cases.
\end{proposition}
\begin{proof}
By Theorem \ref{blan} $Y$ inherits a $\mathbb{C}^{*}$-action with isolated fixed points, and by \cite{BB} $X$ is rational, hence $Y$ is rational with $b_{2}(Y)=1$, implying it is a smooth Fano $3$-fold. By  \cite[Theorem 1.1.2]{KPS} $Y$ is contained in one of the families $\mathbb{CP}^{3},Q,V_{5}$ or $V_{22}$. Hence, combining Lemma \ref{projp}, Lemma \ref{quadp} and Lemma \ref{ex} gives the result
\end{proof}

\subsection{Analysis of blow-downs to singular points}

Next, we deal with the final remaining two extremal contraction, a contraction of a $\mathbb{CP}^{2}$ with normal bundle $\mathcal{O}(-2)$, to a cyclic quotient singularity of type $\frac{1}{2}(1,1,1)$ i.e. $\mathbb{C}^{3}/\pm 1$ and the blow down of a quadric cone $P$ to a double cDV point. We will need the facts that the ODP singularity has Gorenstein index $2$ and is terminal.  We will also need that the double cDV point is Gorenstein and terminal.

\begin{lemma} \label{borfano}
Consider the class of smooth $3$-folds with a $\mathbb{C}^{*}$-action with $6$ fixed points having an extremal contraction of type 3 or 4. Then, there are a finite number of such $3$-folds up to AC-equivalence.
\end{lemma}
\begin{proof}
Let $X$ be such a threefold and $\phi: X \rightarrow Y $ the extremal contraction. Since by the Bialynicki-Birula decomposition \cite[Theorem 4.3]{BB} $X$ is rational, so $Y$ is rational. Since $b_{2}(Y) = 1$, $Y$ is Fano by the Kleiman ampleness criterion. Since in both cases the singularities are terminal and Gorenstein index $1$ or $2$, $Y$ fits into a bounded family of possible cases by  \cite[Theorem 1.1]{B}. Hence, blowing up the singular point, there is finitely many possibilities for $X$ up to AC-equivalence. \end{proof}

\begin{remark}
We note that there is a more down to earth, but longer proof of Lemma \ref{borfano}. One may use the bounds on $(-K_{Y})^{3}$ and $c_{2}(Y).-K_{Y}$ proved in \cite{B} directly to prove that there are finitely many possibilities for the cohomology ring and characteristic classes of $X$. Then, the boundedness of the AC-equivalence classes follows from the Jupp-Jubr-Wall classification theorem for simply connected $6$-manifolds.
\end{remark}




\subsection{Concluding the proof of Theorem \ref{newmain} and Theorem \ref{twotorus}}
Here, we conclude the proof of Theorem \ref{newmain} by putting together the analysis of the different extremal contractions.

\begin{proof}[Proof of Theorem \ref{newmain}]

Suppose $X$ is a smooth  $3$-fold with a $\mathbb{C}^{*}$-action with $6$ fixed points. Since by Lemma \ref{betti}, $X$ is rational, $K_{X}$ is not nef, hence $X$ has a Mori extremal contraction. We can assume that the Mori extremal contraction $\phi: X \rightarrow Y$ is birational, since otherwise it is a Mori fibre space as required. We have to deal with the extremal contractions of type 1-5 from Theorem \ref{main}. By Lemma \ref{conres}, type $5$ does not occur. For extremal contraction of the types 1.-4.. the result follows from Lemma \ref{pointb}, Proposition \ref{curveb} and Lemma \ref{borfano} respectively.
\end{proof}

We also prove Theorem \ref{twotorus}:

\begin{proof}[Proof of the Theorem \ref{twotorus}] 

The proof is essentially is similar to the proof of Theorem \ref{newmain}, with a some additional remarks. We again analyze the possible divisorial extremal contractions, but at each stage we exclude $3$-folds not having a $\mathbb{C}^{*} \times \mathbb{C}^{*}$-action. The action of a generic subtorus has isolated fixed points so we may apply Lemma \ref{conres}, excluding extremal contractions of type 5 from Theorem \ref{main}.  We deal with the remaining types of extremal contractions.

\textit{Point Blow-ups and curve blow ups}  Let $\phi : X \rightarrow Y$ be a contraction to a smooth point or curve, similarly to the proof of Theorem \ref{newmain}, $Y$ is a smooth Fano $3$-fold. By Theorem \ref{blan}, $Y$ admits an effective $\mathbb{C}^{*} \times \mathbb{C}^{*}$-action. By  \cite[Theorem 1.1.2]{KPS}, $Y$ is isomorphic to either $\mathbb{CP}^{3}$ or  $Q$, a smooth quadric $3$-fold. For point blow-ups we note that $Bl_{p}(\mathbb{CP}^{3})$ is toric and otherwise we have $Bl_{p}(Q)$ as required. 

For smooth curve blow-ups, as was noted in Remark \ref{twoone} and Remark \ref{twotwo}, the $3$-folds appearing Example \ref{cptc} an Example \ref{quadc} with $n \geq 2$ do not have $\mathbb{C}^{*} \times \mathbb{C}^{*}$-actions, hence may be excluded. For $n=1 $, $Bl_{C_{1}}(\mathbb{CP}^{3})$ and $Bl_{C'_{1}}(Q)$ are smooth Fano $3$-folds with $b_{2}=2$ hence Mori fibre spaces. Hence, there are no exceptional cases coming from smooth curve blow-ups.

\textit{Extremal contractions of type 3 and 4.}  Suppose  $\phi: X \rightarrow Y$ is an extremal contraction of type 3 or 4. We note that $Y$ is Fano by the Kleimann ampleness criteria. Also by Theorem \ref{blan}, the $\mathbb{C}^{*} \times \mathbb{C}^{*}$-action descends to an effective $\mathbb{C}^{*} \times \mathbb{C}^{*}$-action on $Y$. $Y$ is $\mathbb{Q}$-factorial since it is the image of a divisorial extremal contraction  with relative Picard number $1$ from a smooth $3$-fold \cite[Section 2.1]{CT}. Furthermore, the singularities appearing in both type 3 and 4 extremal contractions are terminal, and $Y$  has Picard rank $1$, $Y$ fits into the big table of \cite[Theorem 1.1]{BHHN}. Since the Quadric $3$-fold is the only entry in the table which is Gorenstien (i.e. Gorenstein index 1) $\phi$ cannot be an extremal contraction of type 3, hence we suppose it is of type 4.

 Note that the equation $K_{X} - \phi^{*}(K_{Y}) = \frac{1}{2}E$ implies $K_{Y}^{3}$ cannot be an integer. Combining this with the fact Gorenstein index  of $Y$ is $2$, there is only one possibility for $Y$ up to isomorphism, namely $3$-fold 11. from the table in \cite[Theorem 1.1]{BHHN}, as required. \end{proof}

\section{Mori Fibre spaces}\label{mfa}

In this section we conclude by discussing the remaining case of smooth $3$-dimensional Mori fibre spaces. Algebraic torus actions on such spaces are studied comprehensively in other places,  we just note a few well-known facts about the topology of such spaces. Firstly we recall  a very basic topological restriction of $3$-dimensional Mori fibre spaces.

\begin{lemma} \label{cubez}
Let $X$ be a smooth, $3$-dimensional Mori fibre space with $b_{2}(X)>1$. Then there exists a non-zero $\alpha \in H^{2}(M,\mathbb{Z}) $ such that $\int_{X} \alpha^{3} = 0$. 
\end{lemma}
\begin{proof}
Let $\phi: X \rightarrow Y$ be a Mori fibre space. Since $b_{2}(X)>1$, some curves are not contracted by $\phi$, so $\dim(Y)>0$. Also, since $Y$ is projective, we may take $\alpha = \phi^{*}([\eta])$, where $\eta$ is a K\"{a}hler form on $Y$. Note that $\alpha \neq 0$ since we may choose a divisor in $X$, which has to project to either a curve or divisor in $Y$. 
\end{proof}

\begin{lemma} \label{square}
Suppose that $X$ is a smooth $3$-fold with is a Mori extremal contraction $\phi : X \rightarrow Y$ where $\dim(Y)=1$. Then $\Delta(X) = 0$
\end{lemma}
\begin{proof}
By \cite{M} $Y$ is a smooth projective curve. The fibre class $F \in H^{2}(X,\mathbb{Z})$ satisfies $F^{2}=0$. By \cite[Proposition 5]{OV}, the existence of such an element is equivalent to the vanishing of $\Delta(X)$.
\end{proof}

In the following Lemma, which is similar to \cite[Lemma 2.4]{CPS}, we classify nodal curve which are invariant by a $\mathbb{C}^{*}$-action on $\mathbb{CP}^{2}$. This will be used below to analyze discriminant curves of conic bundles.

\begin{lemma} \label{plane}
Suppose that $C \subset \mathbb{CP}^{2}$ is a (possibly reducible) curve in $\mathbb{CP}^{2}$ with at worst nodal singularities and $C$ is invariant by a non-trivial $\mathbb{C}^{*}$-action on $\mathbb{CP}^{2}$. Then $\deg(C) \leq 3$, in particular either \begin{enumerate}
\item $C$ is a union of up to $3$ lines in general position.
\item $C$ is an irreducible quadric.
\item $C$ is a union of a line and an irreducible quadric.
\end{enumerate}
\end{lemma}
\begin{proof}
Up to a change of co-ordinates we may write the $\mathbb{C}^{*}$-action on $\mathbb{CP}^{2}$ as $z.[x_{0}:x_{1}:x_{2}] = [x_{0}:x_{1}^a:x_{2}^b]$ where, without loss of generality $0 \leq a \leq b$ are integers. Consider an invariant curve $C$, then assuming $C$ is not a co-ordinate line, then up to an automorphism, $C$ is given parametrically by $[X^{b}: X^{b-a}Y^{a}:Y^{b}]$. Then, one may check that this curve is smooth if and only if it is a line or $a=1$, $b=2$ i.e. it is a smooth conic, and for $b>2$ the curve has non-nodal singularities. Furthermore, any invariant conic contains both $p_{\min} $ and $p_{\max}$ (i.e. the fixed points where the action has both weights positive resp. negative), with any pair of them meeting tangentially at these fixed points. The result follows. 
\end{proof}

\begin{lemma} \label{disc}
Suppose that $X$ is a smooth $3$-fold with a $\mathbb{C}^{*}$-action with $6$ fixed points. Suppose there is a Mori extremal contraction $\phi : X \rightarrow Y$ where $\dim(Y)=2$. Then $X$ is a conic bundle over $\mathbb{CP}^{2}$ with discriminant curve $\Gamma \subset \mathbb{CP}^{2}$ having degree at most $3$.
\end{lemma}
\begin{proof}
By Theorem \ref{blan}, the action descends to an action on $Y$ with isolated fixed points. Hence, $Y$ is a torc surface \cite{Ka} with $b_{2}(Y)=1$, implying that $Y \cong \mathbb{CP}^{2}$. We note that since $\phi$ is equivariant, the discriminant curve is  invariant by the induced action on $\mathbb{CP}^{2}$. It is also nodal by \cite[Corollary 3.3.3]{P}. Hence, by Lemma \ref{plane}, the discriminant curve has degree at most $3$.
\end{proof}

We note that the discriminant curve of smooth rational conic bundles over $\mathbb{CP}^{2}$ has degree at most $5$ \cite[Theorem 9.1]{P}, hence the above lemma just excludes the cases of degree $4$ or $5$. We note also that if $X$ is spin i.e. $-K_{X}$ is numerically divisible by $2$ then the discriminant curve has to be empty (see \cite[Theorem 10]{ST}), since clearly a fibre cannot be a union of two lines in $\mathbb{CP}^{2}$. Then by the vanishing of Brauer group of $\mathbb{CP}^{2}$, the conic bundle is of the form $\mathbb{P}(E)$ for some rank $2$ bundle $E$ over $\mathbb{CP}^{2}$.

The following Fano $3$-fold shows that degree $3$ discriminant curves do occur:

\begin{example} \cite[Lemma 10.2]{CPS} \label{ann}
Consider the $3$-fold $X = \{x_0 y_0^2 + x_1 y_1^2 + x_3 y_2^2 =0 \}  \subset \mathbb{P}^2 \times \mathbb{P}^2$.   Consider the $\mathbb{C}^{*} \times \mathbb{C}^{*}$ action on $X$ given by $$(z_1,z_2)([x_0:x_1:x_2])([y_0:y_1:y_2]) =  ([z_1^2 x_0:  z_2^2 x_1:x_2])([z_1^{-1}y_0:  z_{2}^{-1} y_1:y_2]) .$$ Then this action has $6$ isolated fixed points, furthermore, the projection $\pi_1$ to the first factor gives $X$ the structure of a conic bundle, whose discriminant curve is the union of the three coordinate lines in $\mathbb{CP}^2$.

\end{example}
We note that although the above $3$-folds has a structure of a equivariant conic bundle over $\mathbb{CP}^2$ with non-empty discriminant curve, the projection $\pi_2$ is a $\mathbb{CP}^1$ bundle, giving $X$ the structure of a projective bundle $X= \mathbb{P}(E)$ where $c_{1}(E) = 2$, $c_{2}(E) = 4$. In particular $\Delta(\mathbb{P}(E)) = -12$. We are unaware if there is an example which is not diffeomorphic to a $\mathbb{CP}^1$-bundle over $\mathbb{CP}^2$.

\subsubsection{Proof of Corollary \ref{dis}}
Finally, we prove Corollary \ref{dis} combining results of the previous section.
\begin{proof}[Proof of Corollary \ref{dis}]
If $X$ is one of the $3$-folds in Example \ref{cptc} and Example \ref{quadc} then $\Delta(X) $ is bounded above, as follows from Corollary \ref{blowupsequence}. Hence, by Theorem \ref{newmain} it remains to prove the statement when $X$ is Mori fibre space $\phi: X \rightarrow Y$. By Lemma \ref{square}, if $\dim(Y)=1$ then $\Delta(X)=0$. In the remaining case that $\dim(Y)=2$ i.e. $X$ is a conic bundle, then Lemma \ref{disc} gives that the degree of the discriminant curve is at most $3$.
\end{proof}

We also give a purely topological statement, showing that cubic intersection forms of smooth $3$ folds with a $\mathbb{C}^{*}$-actions with $6$ fixed points lie in a very restricted subset of the possible integral cubic forms on two variables.
\begin{corollary} \label{distop}
There is a constant $K$ such that any smooth $3$-fold $X$ having a $\mathbb{C}^{*}$-action with $6$ fixed points, satisfies one of the following two conditions:
\begin{enumerate}
\item There exists non-zero $\alpha \in H^{2}(X,\mathbb{Z})$ such that $\int_{X} \alpha^{3}=0$
\item $\Delta(X) <K$. 

\end{enumerate}
\end{corollary}
\begin{proof}
This follows from Corollary \ref{dis} and Lemma \ref{cubez}.
\end{proof}

We remark that the condition of Case 1. is a rarely satisfied by smooth rational $3$-folds. For a simple example consider the blow-up of the quadric $3$-fold $Q$ in a point. Then in the natural integral basis of $H^2(Bl_{p}(Q),\mathbb{Z}))$, the cubic form is  $(xa+by)^{3}=2x^{3}+y^{3}$, hence $(xa+yb)^{3}=0$ has no non-zero integer solutions. The reader may play with standard constructions to convince themselves that this is the typical situation.

Universität zu Köln: Mathematisches Institut.
Weyertal 86-90
50931 Köln
Germany

Email: nlindsay@math.uni-koeln.de

\end{document}